\documentclass{amsart}

\usepackage{amsfonts,amssymb}
\usepackage{amsmath}
\usepackage{algorithm}

\newtheorem{thm}{Theorem}

\newtheorem{definition}{Definition}
\newtheorem{example}{Example}

\newcommand{\bx}{\mathbf x}
\newcommand{\bw}{\mathbf w}

\begin{document}

\title[A new modified Newton iteration ]{A new modified Newton iteration for computing nonnegative $Z$-eigenpairs  of nonnegative tensors}

\author{Chun-Hua Guo}
\address{Department of Mathematics and Statistics, 
University of Regina, Regina,
SK S4S 0A2, Canada} 
\email{chun-hua.guo@uregina.ca}

\author{Wen-Wei Lin}
\address{Department of Applied Mathematics, National Yang Ming Chiao Tung University, Hsinchu 300, Taiwan}
\email{wwlin@math.nctu.edu.tw}

\author{Ching-Sung Liu}
\address{Department of Applied Mathematics, National University of Kaohsiung, Kaohsiung 811, Taiwan}
\email{chingsungliu@nuk.edu.tw}    



\begin{abstract}

We propose a new modification of Newton iteration for finding  some nonnegative $Z$-eigenpairs of a nonnegative tensor.
The method has local quadratic convergence to a nonnegative
 eigenpair of a nonnegative tensor, under the usual assumption guaranteeing the local quadratic
convergence of the original Newton iteration.

\end{abstract}

\maketitle

\section{Introduction}
\label{sec1}

A real-valued $m$th-order $n$-dimensional tensor $\mathcal{A}$ has the form
\begin{equation*}
\mathcal{A}=(A_{i_{1}i_{2}\ldots i_{m}})\text{, }\quad A_{i_{1}i_{2}\ldots
i_{m}}\in \mathbb{R}\text{, }\quad 1\leq i_{1},i_{2},\ldots ,i_{m}\leq n.
\end{equation*}
The set of all such tensors is denoted by $\mathbb{R}^{[m,n]}$, and the set of all nonnegative tensors $\mathcal{A}\in \mathbb{R}^{[m,n]}$, for which $A_{i_{1}i_{2}\ldots
i_{m}}\geq 0$ for all $i_{1}, i_{2}, \ldots, 
i_{m}$, is denoted by  $\mathbb{R}_{+}^{[m,n]}$. 
We use $x_i$ or $(\mathbf{x})_i$  to represent the $i$th element of a column vector $\mathbf{x}$.

\begin{definition}[\protect\cite{Q05,CZ13}]
Let $\mathcal{A}\in \mathbb{R}^{[m,n]}$. 
 We say that $(\mathbf{x}, \lambda )\in \left( \mathbb{R}^{n}\backslash \{0\}\right)  \times \mathbb{R}$ is a $Z$-eigenpair
(eigenvector-eigenvalue) of $\mathcal{A}$ if
 \begin{equation}
\mathcal{A}\mathbf{x}^{m-1}=\lambda \mathbf{x}, \quad \|\mathbf{x}\|=1,   \label{eq: NEP}
\end{equation}
where $(\mathcal{A}\mathbf{x}^{m-1})_i=\sum_{i_{2},\ldots
,i_{m}=1}^{n}A_{ii_{2}\ldots i_{m}}x_{i_{2}}\ldots x_{i_{m}}$ for $i=1, \ldots, n$.  
\end{definition}

The vector norm in \eqref{eq: NEP} can be $2$-norm or $1$-norm.  A $Z$-eigenpair is called a $Z_2$-eigenpair when $2$-norm is used, and 
  a $Z_1$-eigenpair when $1$-norm is used. 
It is shown in \cite{CZ13} that, for $\mathbf{x}$ with $\|\mathbf{x}\|_1=1$, 
$(\mathbf{x}, \lambda)$ is a $Z_1$-eigenpair if and only if 
$\left ( \frac{\mathbf{x}}{\|\mathbf{x}\|_2},  \frac{\lambda}{\|\mathbf{x}\|_2^{m-2}} \right )$
 is a $Z_2$-eigenpair. In this paper, we are interested in $Z_1$-eigenpairs and a $Z_1$-eigenpair will be referred to as a $Z$-eigenpair or simply an eigenpair. 

A modified Newton iteration (MNI) proposed in \cite{GLL19} can usually 
find some nonnegative eigenpairs of a nonnegative tensor using different initial vectors. A local quadratic convergence result is proved there for positive eigenpairs. 
(All nonnegative $Z$-eigenpairs are positive when the tensor is irreducible \cite{CPZ13}). It is observed in \cite{GLL19} that MNI can also compute nonnegative eigenpairs with some zero components in the eigenvectors, 
but the convergence is much slower.  
Recently, another modified Newton iteration (called PNI) is proposed in \cite{BLLX21}, where a projection is used at a later stage than in \cite{GLL19}. 
PNI can significantly speed up the convergence in computing nonnegative eigenpairs with some zero components in the eigenvectors. 
In this paper, we will address some issues associated with PNI, particularly the determination of $\lambda_k$ in the sequence $(\lambda_k, \mathbf{x}_k)$ approximating a nonnegative eigenpair of the tensor, and present a new modification of the Newton iteration.

\section{A new modified Newton iteration}
\label{sec2}

To compute an eigenpair $(\mathbf{x}_{\ast}, \lambda_{\ast})$ with
$\mathbf{x}_{\ast}\ge 0$ and $\|\mathbf{x}_{\ast}\|_1=\mathbf{e}^T\mathbf{x}
_{\ast}=1$, where $\mathbf{e}=[1,\ldots ,1]^{T},$ we can apply Newton's method to $\mathbf{f}(\mathbf{x},\lambda )=0$, where 

\begin{equation} 
\mathbf{f}(\mathbf{x},\lambda )=\left[
\begin{array}{c}
\lambda \mathbf{x}-\mathcal{A}\mathbf{x}^{m-1}            \\
\mathbf{e}^{T}\mathbf{x}-1
\end{array}
\right].  \label{eq:Fx}
\end{equation}
The Jacobian of $\mathbf{f}(x,\lambda )$ is given by
\begin{equation}
\mathbf{Jf}(\mathbf{x},\lambda )=\left[
\begin{array}{cc}
\lambda I-T(\mathbf{x}) & \mathbf{x} \\
\mathbf{e}^{T} & 0
\end{array}
\right] ,  \label{eq: graF}
\end{equation}
where the entries of $T(\mathbf{x})$ are
\begin{equation*}  
T(\mathbf{x})_{ij}=\frac{\partial}{\partial x_j}\left( \mathcal{A}\mathbf{x}^{m-1}\right)_i.
\end{equation*}

For a given approximation $(\widehat{\mathbf{x}}_{k},\widehat{\lambda }_{k})$  to $(\mathbf{x}_*, \lambda_*)$, 
Newton's method determines the next approximation $(\widehat{\mathbf{x}}_{k+1},\widehat{\lambda }_{k+1})$ by

\begin{align}
\left[
\begin{array}{cc}
\widehat{\lambda}_k I-T({\widehat{\mathbf{x}}}_{k}) & \widehat{\mathbf{x}}_{k} \\
{\mathbf{e}}^{T} & 0
\end{array}
\right] \left[
\begin{array}{c}
\mathbf{d}_{k} \\
\delta _{k}
\end{array}
\right] & =\left[
\begin{array}{c}
\widehat{\lambda }_{k} \widehat{\mathbf{x}}_{k} -\mathcal{A}\widehat{\mathbf{x}}_{k}^{m-1}    \\
{\mathbf{e}}^{T}\widehat{\mathbf{x}}_{k}-1
\end{array}
\right],  \label{eq:step1} \\
\widehat{\mathbf{x}}_{k+1}& =\widehat{\mathbf{x}}_{k}\,-\mathbf{d}_{k},
\label{eq:step2} \\
\widehat{\lambda }_{k+1}& =\widehat{\lambda }_{k}-\delta _{k}, 
\label{eq:step3}
\end{align}
assuming that 
\begin{equation}\label{coef}
\left[
\begin{array}{cc}
\widehat{\lambda}_k I-T({\widehat{\mathbf{x}}}_{k}) & \widehat{\mathbf{x}}_{k} \\
{\mathbf{e}}^{T} & 0
\end{array}
\right] 
\end{equation}
is nonsingular. 

To put our new modified Newton iteration in a proper setting, we give below a brief description of MNI in \cite{GLL19} and PNI in \cite{BLLX21}. 
We also mention that a Newton-based method has been studied in \cite{JWN18} for computing $Z$-eigenpairs of symmetric tensors. 

When  ${\mathbf{e}}^{T}\widehat{\mathbf{x}}_{k}=1$ and $\widehat{\lambda}_k I-T({\widehat{\mathbf{x}}}_{k})$ is nonsingular, 
$(\widehat{\mathbf{x}}_{k+1},\widehat{\lambda }_{k+1})$ is also given by (\cite{GLL19,BLLX21})
\begin{align}
\widehat{\mathbf{x}}_{k+1}& =
\frac{1}{m-1}\left ( (m-2) \widehat{\mathbf{x}}_{k}+\frac{1}{\mathbf{e}^{T}
\widehat{\mathbf{w}}_{k}}\widehat{\mathbf{w}}_{k} \right ),  \label{eq:newtonup} \\
\widehat{\lambda }_{k+1}& =\frac{1}{m-1}\left (
\widehat{\lambda }_{k}-\frac{1}{{\mathbf{e}}^{T}\widehat{\mathbf{w}}
_{k} }\right ),   \label{ndown}
\end{align}
where 
\begin{equation}  \label{vw}
\widehat{\mathbf{w}}_{k}=\left (\widehat{\lambda}_k I-T({\widehat{\mathbf{x}}}_{k})\right )^{-1}\widehat{\mathbf{x}}_{k}.
\end{equation}
That $\mathbf{e}^{T}\widehat{\mathbf{w}}_{k}\ne 0$ is equivalent to the nonsingularity of \eqref{coef}. 

In \cite{GLL19}, a projection is applied to the vector $\widehat{\mathbf{w}}_{k}$ used in \eqref{eq:newtonup} (but not to 
the vector $\widehat{\mathbf{w}}_{k}$ used in \eqref{ndown}): 
\begin{equation}
{\mathbf{w}}_k=\left\{
\begin{array}{ll}
\max (\widehat{\mathbf{w}}_{k}, \mathbf{0}),  & \text{if }        |\max \widehat{\mathbf{w}}_{k}|>|\min \widehat{\mathbf{w}}_{k}|,                   \\
 \min (\widehat{\mathbf{w}}_{k}, \mathbf{0}),             &  \text{if }        |\max \widehat{\mathbf{w}}_{k}|\le |\min \widehat{\mathbf{w}}_{k}|,            \\
\end{array} 
\right.  \label{wupdate}
\end{equation}
where $\max$ and $\min$ are taken elementwise. 
With $\widehat{\mathbf{w}}_{k}$ in \eqref{eq:newtonup} replaced by ${\mathbf{w}}_k$, the vector $\widehat{\mathbf{x}}_{k+1}$ 
from \eqref{eq:newtonup} is denoted by ${\mathbf{x}}_{k+1}$. This gives a new approximation to ${\mathbf{x}}_{\ast}$: 
${\mathbf{x}}_{k+1}>0$ with ${\mathbf{e}}^{T}{\mathbf{x}}_{k+1}=1$. Note that ${\mathbf{x}}_{k+1}$ can be obtained even when  $\mathbf{e}^{T}\widehat{\mathbf{w}}_{k}= 0$. 

For a pair of $n$-vectors $\mathbf{v}$ and $\mathbf{w}$ with $v_i\ne 0$ for each $i$, we define
\begin{equation}\label{maxmin}
\max \left( \frac{\mathbf{w}}{\mathbf{v}}\right) =\underset{i}{\max }
\left (\frac{w_i}{v_i}\right ) ,\text{ \ }\min \left(
\frac{\mathbf{w}}{\mathbf{v}}\right) =\underset{i}{\min }\left (\frac{w_i}{v_i}\right ). 
\end{equation}

A new approximation $\lambda_{k+1}$  to $\lambda_{\ast}$ is suggested in \cite{GLL19} to be any value in the interval $[\underline{\lambda }_{k+1}, \overline{\lambda }_{k+1}],$
where 
\begin{equation*}  
\underline{\lambda }_{k+1} =\min \left( \frac{\mathcal{A}{\mathbf{x}}
_{k+1}^{m-1}}{{\mathbf{x}}_{k+1}}\right), \quad 
\overline{\lambda }_{k+1} =\max \left( \frac{\mathcal{A}{\mathbf{x}}
_{k+1}^{m-1}}{{\mathbf{x}}_{k+1}}\right), 
\end{equation*}
such that ${\lambda}_{k+1} I-T({\mathbf{x}_{k+1}})$ is not singular or nearly singular. 

The following modified Newton iteration (Algorithm \ref{alg:MNI}) has been presented in \cite{GLL19} for finding a nonnegative eigenpair of a  nonnegative tensor $\mathcal{A}$. 

\begin{algorithm}
\begin{enumerate}
  \item   Given $\bx_0 > 0$ with $\Vert \bx_0\Vert_1 =1$, and ${\sf tol}>0$.
\item Compute  $\overline{\lambda}_{0}= \max \left( \frac{\mathcal{A}\mathbf{x}_{0}^{m-1}}{\mathbf{x}_{0}}\right)$ and 
 $\underline{\lambda}_{0}= \min \left( \frac{\mathcal{A}\mathbf{x}_{0}^{m-1}}{\mathbf{x}_{0}}\right)$. 
\item  Choose $\lambda_{0}\in [\underline{\lambda }_{0}, \overline{\lambda }_{0}]$
such that ${\lambda}_{0} I- T({\mathbf{x}_{0}})$ is nonsingular. 
  \item   {\bf for} $k =0,1,2,\dots$    {\bf until} $\left\Vert \mathcal{A}\mathbf{x}_{k}^{m-1} - {\lambda }_{k}\mathbf{x}_{k} \right\Vert_1  <{\sf tol}$.
  \item   \quad Solve the linear system $\left ( {\lambda}_k I-T(\mathbf{x}_{k})\right) \widehat{\mathbf{w}}_{k}=\mathbf{x}_{k}$.
  \item   \quad Determine the vector $\bw_{k}$ by \eqref{wupdate}.
  \item   \quad Compute the vector $\widetilde{\mathbf{x}}_{k+1} =(m-2)\mathbf{x}_{k}\,+  \mathbf{w}_k/  (\mathbf{e}^T  \mathbf{w}_{k})$.
  \item   \quad Normalize the vector $\widetilde{\mathbf{x}}_{k+1}$:   $\bx_{k+1}= \widetilde{\mathbf{x}}_{k+1}/\Vert \widetilde{\mathbf{x}}_{k+1}\Vert_1$.
  \item   \quad Compute $\overline{\lambda }_{k+1} =\max \left( \frac{\mathcal{A}\mathbf{x}_{k+1}^{m-1}}{\mathbf{x}_{k+1}}\right)$ and
  $\underline{\lambda }_{k+1} =\min \left( \frac{\mathcal{A}\mathbf{x}_{k+1}^{m-1}}{\mathbf{x}_{k+1}}\right)$.
 \item \quad Choose $\lambda_{k+1}\in [\underline{\lambda }_{k+1}, \overline{\lambda }_{k+1}]$
such that ${\lambda}_{k+1} I- T({\mathbf{x}_{k+1}})$ is nonsingular. 
\end{enumerate}
\caption{Modified Newton iteration (MNI)}
\label{alg:MNI}
\end{algorithm}

The default value for $\lambda_0$ in line 3 of Algorithm \ref{alg:MNI}  is $\lambda_{0}=\overline{\lambda }_{0}$. 
When  ${\bf e}^T\widehat{\mathbf{w}}_{k}\ne 0$ for $\widehat{\mathbf{w}}_{k}$   in line 5, we use \eqref{ndown} to get 
$$
\widehat{\lambda }_{k+1} =\frac{1}{m-1}\left (
{\lambda }_{k}-\frac{1}{{\mathbf{e}}^{T}\widehat{\mathbf{w}}_{k} }\right ). 
$$
The default value for $\lambda_{k+1}$ in line 10 is then given by 
\begin{equation*}
\lambda_{k+1}=\left \{ \begin{array}{ll}
 \overline{\lambda }_{k+1} & \mbox{ if }   {\bf e}^T\widehat{\mathbf{w}}_{k}= 0 \mbox{ or } \widehat{\lambda }_{k+1}> \overline{\lambda }_{k+1},\\
\underline{\lambda }_{k+1} & \mbox{ if }  \widehat{\lambda }_{k+1}< \underline{\lambda }_{k+1},\\
\widehat{\lambda }_{k+1} & \mbox{  if }     \widehat{\lambda }_{k+1}       \in [\underline{\lambda }_{k+1}, \overline{\lambda }_{k+1}].
\end{array} \right. 
\end{equation*}
If the chosen $\lambda_k$ is such that ${\lambda}_{k} I-T({\mathbf{x}_{k}})$ is singular or nearly singular, we can always adjust it within the interval $[\underline{\lambda}_k, \overline{\lambda}_k]$. 
Note that we have $\underline{\lambda }_{k}<\overline{\lambda }_{k}$ unless $(\mathbf{x}_k, \lambda_k)$ is already an eigenpair. 

The PNI in \cite{BLLX21} can be described through a comparison with MNI. In PNI, line 6 of MNI is not performed. So the  $\mathbf{w}_k$ in line 7 is still the $\widehat{\mathbf{w}}_{k}$ 
in line 5. It is assumed in \cite{BLLX21} that  ${\bf e}^T\widehat{\mathbf{w}}_{k}\ne  0$. Then 
${\bf e}^T\widetilde{\mathbf{x}}_{k+1}=m-1>0$ and $\widetilde{\mathbf{x}}_{k+1}$ has some positive components. In PNI, all negative components of   $\widetilde{\mathbf{x}}_{k+1}$ 
are replaced by $0$. That is, the update $\widetilde{\mathbf{x}}_{k+1}\leftarrow \max (\widetilde{\mathbf{x}}_{k+1}, \mathbf{0})$ is performed before line 8 is performed. 
In this way, PNI can potentially find nonnegative eigenvectors with some zero components much more quickly. But it has also created a new problem: in line 9 of MNI the definition in \eqref{maxmin} must be modified. 
For a pair of nonnegative vectors $\mathbf{w}$ and $\mathbf{v}\ne \mathbf{0}$, it is defined in \cite{BLLX21} that 
\begin{eqnarray*}
&&\max \left( \frac{\mathbf{w}}{\mathbf{v}}\right) =\left \{ \begin{array}{ll} 
\max \left \{ \underset{i\in S_2}{\max }
\left (\frac{w_i}{v_i}\right ) ,  \underset{i\in S_1 \setminus (S_1\cap S_2)}{\max }
\left (w_i \right ) 
\right \}  & \quad \mbox{if }   S_1 \setminus (S_1\cap S_2) \ne \emptyset,\\
 \underset{i\in S_2}{\max }\left (\frac{w_i}{v_i}\right )  &  \quad \mbox{if }   S_1 \setminus (S_1\cap S_2) = \emptyset, 
\end{array} \right.        \\
&& \min \left( \frac{\mathbf{w}}{\mathbf{v}}\right) =\left \{ \begin{array}{ll} 
0  &  \quad \mbox{if }   S_1 \setminus (S_1\cap S_2) \ne \emptyset,\\
 \underset{i\in S_2}{\min }\left (\frac{w_i}{v_i}\right )  & \quad  \mbox{if }   S_1 \setminus (S_1\cap S_2) = \emptyset, 
\end{array} \right. 
\end{eqnarray*}
where $S_1$ and $S_2$ are the index sets of all nonzero elements of $\mathbf{w}$ and $\mathbf{v}$,   respectively. 
Note that the above definition reduces to definition \eqref{maxmin} when $\mathbf{v}> \mathbf{0}$. 

In PNI \cite{BLLX21}, $\lambda_{k+1}$ is determined using
\begin{equation}\label{lamBLLX}
\lambda_{k+1}=\left \{ \begin{array}{ll}
\widehat{\lambda }_{k+1}+\beta_{k+1} \left ( \overline{\lambda }_{k+1}- \widehat{\lambda }_{k+1}\right )    & \mbox{ if }   \widehat{\lambda }_{k+1}\le         (\underline{\lambda }_{k+1} +\overline{\lambda }_{k+1})/2,\\
\widehat{\lambda }_{k+1}+\beta_{k+1} \left ( \underline{\lambda }_{k+1}- \widehat{\lambda }_{k+1}\right )    & \mbox{ if }   \widehat{\lambda }_{k+1}>         (\underline{\lambda }_{k+1} +\overline{\lambda }_{k+1})/2,
\end{array} \right. 
\end{equation}
where $0\le \beta_{k+1}\le 1$ is a given constant such that $\lambda_{k+1}I-T(\mathbf{x}_{k+1})$ is nonsingular. 
It is suggested in \cite{BLLX21} that $\beta_{k+1}$ be adjusted such that $\lambda_{k+1}\in [\underline{\lambda }_{k+1}, \overline{\lambda }_{k+1}]$
to get a better approximation $\lambda_{k+1}$  to $\lambda_{\ast}$. 

To prove local quadratic convergence of PNI, it is assumed in \cite{BLLX21} that 
\begin{equation} \label{betaBLLX}
\beta_{k+1}\le \min \left \{1, \frac{\max \left \{ \|\mathbf{x}_{k+1}-\mathbf{x}_{\ast}\|_1, |\widehat{\lambda }_{k+1}- {\lambda }_{\ast}|\right \}}
{\max \left \{ |\overline{\lambda}_{k+1} - \widehat{\lambda}_{k+1}|, |\underline{\lambda }_{k+1} -\widehat{\lambda}_{k+1}|\right \}}\right \}
\end{equation}
in \eqref{lamBLLX}. It is likely that the assumption holds for a small $\beta_{k+1}$ but it is difficult to make sure, since the unknown $(\mathbf{x}_{\ast}, {\lambda }_{\ast})$ is involved. Moreover, a small $\beta_{k+1}$ may not ensure $\lambda_{k+1}\in [\underline{\lambda }_{k+1}, \overline{\lambda }_{k+1}]$. 
This has led us to re-examine the idea of using $\underline{\lambda }_{k+1}$ and $\overline{\lambda }_{k+1}$ to choose $\lambda_{k+1}$.

We consider the following example in \cite{GLL19}. 
\begin{example}
\label{ex3}
Consider  $\mathcal{A}\in \mathbb{R}_{+}^{[4,2]}$ defined by
$$
A_{1111}= 1.1,\ \   A_{2222}=1.2,\ \  A_{1112}=A_{1222}=0.25, \text{ and }A_{ijkl}=0 \text{ elsewhere}.
$$
\end{example}
The tensor has three nonnegative $Z$-eigenpairs, one of them is (rounded to four digits)
$\left (\left [  0.1874, 0.8126     \right ]^T,   0.7923      \right )$. We take $\mathbf{x}$ to be vectors near the eigenvector and determine the corresponding 
intervals $[\underline{\lambda }, \overline{\lambda }]$:
\begin{itemize}
\item  For $\mathbf{x}=[0.19, 0.81]^T$, $[\underline{\lambda }, \overline{\lambda }]=[0.7774, 0.7873]$.
\item For $\mathbf{x}=[0.187, 0.813]^T$, $[\underline{\lambda }, \overline{\lambda }]=[0.7932, 0.7949]$. 
\item For $\mathbf{x}=[0.1875, 0.8125]^T$,  $[\underline{\lambda }, \overline{\lambda }]=[0.7919, 0.7922]$.
\end{itemize}
Note that the eigenvalue  $0.7923$ is not in any of these intervals. This means that  adjusting $\widehat{\lambda }_{k+1}$ from Newton's method 
to $\lambda_{k+1}\in [\underline{\lambda }_{k+1}, \overline{\lambda }_{k+1}]$ will not necessarily give a better approximation to $\lambda_{\ast}$. 

We will assume that $\mathbf{J}\mathbf{f}(\mathbf{x}_{\ast}, \lambda_{\ast})$ is
nonsingular, but ${\lambda}_{\ast} I- T(\mathbf{x}_{\ast})$ may still be singular. 

\begin{example}
\label{ex2}
Consider  $\mathcal{A}\in \mathbb{R}_{+}^{[3,3]}$ defined by
$$
A_{2j3}=A_{3j2}=A_{3j3}=1 \text{ for } j=1,2,3, \text{ and }A_{ijk}=0 \text{ elsewhere}.
$$
One eigenpair of $\mathcal{A}$ is $(\mathbf{x}_{\ast}, \lambda_{\ast})=([1, 0,0]^T, 0)$ and 
$$
\mathbf{J}\mathbf{f}(\mathbf{x}_{\ast}, \lambda_{\ast})=\left [\begin{array}{rrrr}
0 & 0 & 0&1\\
0 &  0 & -1&0\\
0 & -1 & -1& 0\\
1 & 1 & 1& 0
\end{array} \right ]
$$
is nonsingular, while ${\lambda}_{\ast} I- T(\mathbf{x}_{\ast})$ is singular. 
\end{example}

In situations like this, it is more appropriate to determine $(\widehat{\mathbf{x}}_{k+1},\widehat{\lambda }_{k+1})$ using \eqref{eq:step1}--\eqref{eq:step3}, instead of \eqref{eq:newtonup}--\eqref{ndown}. 

These considerations lead to a simple modification of Newton iteration (Algorithm \ref{alg:MPNI}) for finding a nonnegative eigenpair of a  nonnegative tensor $\mathcal{A}$. 
For a vector $\mathbf{x}$ having both positive and negative components,  a nonnegative (and nonzero) vector is obtained as in \cite{GLY15} and \cite{BLLX21}:
$$
\textsf{proj}({\mathbf{x}})=\frac{\max ({\mathbf{x}}, \mathbf{0})}{\|\max ({\mathbf{x}}, \mathbf{0})\|_1}. 
$$

\begin{algorithm}
\begin{enumerate}
  \item   Given $\bx_0 > 0$ with $\Vert \bx_0\Vert_1 =1$, and ${\sf tol}>0$.
\item Compute  $\lambda_{0}= \max \left( \frac{\mathcal{A}\mathbf{x}_{0}^{m-1}}{\mathbf{x}_{0}}\right).$
  \item   {\bf for} $k =0,1,2,\dots$    {\bf until} $\left\Vert \mathcal{A}\mathbf{x}_{k}^{m-1} - {\lambda }_{k}\mathbf{x}_{k} \right\Vert_1  <{\sf tol}$.
  \item \quad  Compute $(\widehat{\mathbf{x}}_{k+1},\widehat{\lambda }_{k+1})$ using \eqref{eq:step1}--\eqref{eq:step3}, 
with  $(\widehat{\mathbf{x}}_{k},\widehat{\lambda }_{k})=({\mathbf{x}}_{k}, {\lambda }_{k})$. 
\item \quad Set ${\mathbf{x}}_{k+1} =\textsf{proj}(\widehat{\mathbf{x}}_{k+1})$. 
  \item \quad Set $\lambda_{k+1}=\max (\widehat{\lambda }_{k+1}, 0)$. 
\end{enumerate}
\caption{Modified projected Newton iteration (MPNI)}
\label{alg:MPNI}
\end{algorithm}

In PNI \cite{BLLX21}, effort is made to choose $\lambda_k$ such that ${\lambda}_{k} I-T({\mathbf{x}_{k}})$ is nonsingular. Actually, the choice of 
$\lambda_k$ should also ensure that 
\begin{equation}\label{coef2}
\left[
\begin{array}{cc}
{\lambda}_k I-T({{\mathbf{x}}}_{k}) & {\mathbf{x}}_{k} \\
{\mathbf{e}}^{T} & 0
\end{array}
\right] 
\end{equation}
is nonsingular. Otherwise, division by zero occurs in the computation of the vector $\widetilde{\mathbf{x}}_{k+1}$ since ${\bf e}^T\widehat{\mathbf{w}}_{k}= 0$ . 
Our Algorithm \ref{alg:MPNI} also requires the nonsingularity of \eqref{coef2}. 
As $\lambda \to \infty$, we have 
$$
\left |\begin{array}{cc}
{\lambda}  I-T({{\mathbf{x}}}_{k}) & {\mathbf{x}}_{k} \\
{\mathbf{e}}^{T} & 0
\end{array} \right |=|\lambda I-T({{\mathbf{x}}}_{k})|(-{\bf e}^T(\lambda I- T({{\mathbf{x}}}_{k}))^{-1}\mathbf{x}_{k})
\sim - \lambda^{n-1}
$$ 
since ${\mathbf{e}}^{T}{\mathbf{x}}_{k}=1$. So the determinant is a polynomial of degree $n-1$. 
Therefore, when the matrix \eqref{coef2} is singular, we can choose a small $\epsilon_k >0$ and use the update 
$\lambda_k\leftarrow \lambda_k+\epsilon_k$ to ensure that the new matrix \eqref{coef2} is nonsingular.

We can easily prove that Algorithm~\ref{alg:MPNI} (MPNI) has local quadratic
convergence under the usual assumption that guarantees the  local quadratic
convergence of the original Newton iteration.

\begin{thm}
\label{quadratic} 
 Let $\left( \mathbf{x}_{\ast }, \lambda_{\ast}\right)$ be
a nonnegative  eigenpair of   a  nonnegative tensor  $\mathcal{A}$, with $\mathbf{J}\mathbf{f}(\mathbf{x}_{\ast}, \lambda_{\ast})$ being nonsingular, and let  $\left\{ (\mathbf{x}_{k}, {\lambda }_{k})\right\}$
be generated by Algorithm~\ref{alg:MPNI}. Suppose that $(\mathbf{x}_{k_0}, {\lambda}_{k_0})$ is sufficiently close to $\left( \mathbf{x}_{\ast }, \lambda_{\ast}\right)$ for some $k_0\ge 0$. 
Then $(\mathbf{x}_{k}, {\lambda }_{k})$ converges
to $(\mathbf{x}_{\ast}, \lambda_{\ast})$ quadratically.
\end{thm}

\begin{proof}
Since $\mathbf{J}\mathbf{f}(\mathbf{x}_{\ast}, \lambda_{\ast})$ is nonsingular, $\mathbf{J}\mathbf{f}(\mathbf{x}_{k}, \lambda_{k})$ is also nonsingular when $(\mathbf{x}_{k}, {\lambda}_{k})$ is sufficiently close to $\left( \mathbf{x}_{\ast }, \lambda_{\ast}\right)$. 
Let $(\widehat{\mathbf{x}}_{k+1},
\widehat{\lambda}_{k+1})$ be obtained by Newton's method as in \eqref{eq:step1}--\eqref{eq:step3}, with $\left (\widehat{\mathbf{x}}_{k}, \widehat{{\lambda }}_{k}\right )=\left (\mathbf{x}_{k},
{\lambda }_{k}\right )$. By a basic result of Newton's method (see \cite[Theorem 5.1.2]{K95} for example), 
there is a constant $c$ such
that
\begin{equation*}
\left\Vert \left[
\begin{array}{c}
\widehat{\mathbf{x}}_{k+1} \\
\widehat{\lambda }_{k+1}
\end{array}
\right] -\left[
\begin{array}{c}
\mathbf{x}_{\ast } \\
\lambda_{\ast}
\end{array}
\right] \right\Vert_1 \leq c \left\Vert \left[
\begin{array}{c}
\mathbf{x}_{k} \\
{\lambda }_{k}
\end{array}
\right] -\left[
\begin{array}{c}
\mathbf{x}_{\ast } \\
\lambda_{\ast}
\end{array}
\right] \right\Vert_1^{2}.  
\end{equation*}
Since $\lambda_{\ast}\ge 0$, we have 
$
|{\lambda }_{k+1}-\lambda_{\ast}|\le |\widehat{\lambda }_{k+1}-\lambda_{\ast}|. 
$
Since ${\bf e}^T\mathbf{x}_k=1$, we have ${\bf e}^T\mathbf{d}_k=0$ from \eqref{eq:step1} and ${\bf e}^T\widehat{\mathbf{x}}_{k+1}= {\bf e}^T\mathbf{x}_k-{\bf e}^T\mathbf{d}_k=1$ from \eqref{eq:step2}. 
It follows from \cite[Lemma 4]{LLV19} that $\|\mathbf{x}_{k+1}-\mathbf{x}_{\ast}\|_1\le \|\widehat{\mathbf{x}}_{k+1}-\mathbf{x}_{\ast}\|_1$.  
Therefore, \begin{equation*}
\left\Vert \left[
\begin{array}{c}
{\mathbf{x}}_{k+1} \\
{\lambda }_{k+1}
\end{array}
\right] -\left[
\begin{array}{c}
\mathbf{x}_{\ast } \\
\lambda_{\ast}
\end{array}
\right] \right\Vert_1 \leq c \left\Vert \left[
\begin{array}{c}
\mathbf{x}_{k} \\
{\lambda }_{k}
\end{array}
\right] -\left[
\begin{array}{c}
\mathbf{x}_{\ast } \\
\lambda_{\ast}
\end{array}
\right] \right\Vert_1^{2}
\end{equation*}
for all $(\mathbf{x}_{k}, {\lambda}_{k})$ sufficiently close to $\left( \mathbf{x}_{\ast }, \lambda_{\ast}\right)$. 
\qed 
\end{proof}

We remark that PNI in \cite{BLLX21} is mathematically equivalent to MPNI here under some favorable assumptions: 
\begin{enumerate}
\item 
$\widehat{\lambda}_k\ge 0$ for all $k$.
\item 
$\widehat{\lambda}_k I -T({\bf x}_k)$ is nonsingular for all $k$ and $\beta_k=0$ is used in \eqref{betaBLLX}. 
\item 
${\bf Jf}({\bf x}_k, \lambda_k)$ is nonsingular for all $k$. 
\end{enumerate}
Therefore, the good numerical performance of PNI reported in \cite{BLLX21} carries over to the MPNI here. But MPNI is simpler and more robust, and its local quadratic convergence can be easily proved without any additional assumptions such as the one in \eqref{betaBLLX}. 
As shown in the proof of Theorem \ref{quadratic}, MPNI has a nice feature: the approximation $({\mathbf{x}}_{k+1}, \lambda_{k+1})$ from MPNI is always better than $(\widehat{\mathbf{x}}_{k+1},\widehat{\lambda }_{k+1})$ from the Newton iteration (starting from the same $({\mathbf{x}}_{k}, {\lambda }_{k})$).

\section*{Acknowledgments} 
C.-H. Guo was supported in part by an NSERC Discovery Grant,
 W.-W. Lin was supported in part by the Ministry of Science and Technology, the National Center for Theoretical Sciences, and ST Yau Center at Chiao-Da in Taiwan, and
C.-S. Liu was supported in part by the Ministry of Science and Technology in Taiwan.\\


\begin{thebibliography}{99}

\bibitem{BLLX21} P.  Bi, W.  Li,  D.  Liu,  M. Xiao, 
The projected Newton iteration approach for
computing the nonnegative $Z$-eigenpairs of
nonnegative tensors, 
CSIAM Trans. Appl. Math.  2 (2021)  376--394. 

\bibitem{CPZ13} K. C. Chang, K. J. Pearson,  T. Zhang, 
Some variational principles for $Z$-eigenvalues
of nonnegative tensors, Linear Algebra Appl.   438 (2013)  4166--4182.

\bibitem{CZ13} K. C. Chang,  T. Zhang, On the uniqueness
and non-uniqueness of the positive $Z$-eigenvector for transition probability
tensors, J. Math. Anal. Appl.  408 (2013) 525--540.

\bibitem{GLY15} D. F. Gleich, L.-H. Lim, Y. Yu, Multilinear PageRank, SIAM J. Matrix Anal. Appl.  36 (2015) 1507--1541.

\bibitem{GLL19} C.-H. Guo,   W.-W. Lin,  C.-S. Liu, 
 A modified Newton iteration for finding nonnegative
$Z$-eigenpairs of a nonnegative tensor, Numer. Algorithms 80 (2019) 595--616.

\bibitem{JWN18} A. Jaffe, R. Weiss, B. Nadler, Newton correction methods for computing real eigenpairs of symmetric tensors,  SIAM J. Matrix Anal. Appl.
39 (2018) 1071–1094.

\bibitem{K95} C. T. Kelley, Iterative Methods for Linear and
Nonlinear Equations, SIAM, Philadelphia, PA, 1995.

\bibitem{LLV19} D.  Liu, W. Li, S.-W. Vong, Relaxation methods for solving the tensor equation arising
from the higher-order Markov chains, Numer. Linear Algebra Appl.  26 (2019)  e2260.

\bibitem{Q05}  L. Qi, Eigenvalues of a real supersymmetric tensor, 
J. Symb. Comput. 40 (2005) 1302--1324.

\end{thebibliography}
\end{document}